\documentclass{llncs}
\usepackage[english]{babel}
\usepackage[utf8]{inputenc}
\usepackage{amsmath,amssymb,amsfonts,graphics,graphicx,enumerate}
\usepackage{tabularx} 
\usepackage{diagbox} 
\usepackage{rotating}
\usepackage{microtype}

\newcommand{\fvs}{\textup{fvs}}
\newcommand{\cfvs}{\textup{cfvs}}
\newcommand{\vc}{\textup{vc}}
\newcommand{\cvc}{\textup{cvc}}
\newcommand{\ds}{\textup{ds}}
\newcommand{\cds}{\textup{cds}}

\newtheorem{observation}{Observation} 

\begin{document}

\title{
The Price of Connectivity for\\ Feedback Vertex Set\thanks{The results in this paper appeared in preliminary form in the proceedings of EuroComb 2013~\cite{BHKP13} and MFCS 2014~\cite{BHKP14}.}}

\author{ 
R\'{e}my Belmonte\inst{1}$^,$\thanks{Supported by the ELC project (Grant-in-Aid for
Scientific Research on Innovative Areas, MEXT Japan).}
\and
Pim van~'t Hof$^{2,}$\thanks{
Supported by the Research Council of Norway (197548/F20) while working at the Department of Computer Science, University of Bergen, Norway.}
\and\\ 
Marcin Kami\'nski\inst{3,}\thanks{Supported by Foundation for Polish Science (HOMING PLUS/2011-4/8) and 
National Science Center (SONATA 2012/07/D/ST6/02432).
}
\and 
Dani\"el Paulusma\inst{4,}\thanks{
Supported by EPSRC (EP/G043434/1) and Royal Society (JP100692).}
}

\institute{
Department of Architecture and Architectural Engineering, Kyoto University, Japan\\
\texttt{remybelmonte@gmail.com}
\and
School of Built Environment, Rotterdam University of Applied Sciences, The Netherlands\\
\texttt{p.van.t.hof@hr.nl}
\and
Institute of Computer Science, University of Warsaw, Poland\\
\texttt{mjk@mimuw.edu.pl}
\and
School of Engineering and Computing Sciences, Durham University, UK\\
\texttt{daniel.paulusma@durham.ac.uk}
}

\maketitle

\begin{abstract}
Let $\fvs(G)$ and $\cfvs(G)$ denote the cardinalities of a minimum feedback vertex set and a minimum connected feedback vertex set of a~graph $G$, respectively. The price of connectivity for feedback vertex set (poc-fvs) for a class of graphs ${\cal G}$ is defined as the maximum ratio $\cfvs(G)/\fvs(G)$ over all connected graphs $G\in {\cal G}$. We study the poc-fvs for graph classes defined by a finite family ${\cal H}$ of forbidden induced subgraphs. We characterize exactly those finite families ${\cal H}$ for which the poc-fvs for  ${\cal H}$-free graphs is upper bounded by a constant.
Additionally, for the case where $|{\cal H}|=1$, we determine exactly those graphs $H$ for which there exists a constant $c_H$ such that $\cfvs(G)\leq \fvs(G) + c_H$ for every connected $H$-free graph $G$, as well as exactly those graphs $H$ for which we can take $c_H=0$.
\end{abstract}

\section{Introduction}
\label{s-intro}

Numerous important graph parameters are defined as the cardinality of a smallest subset of vertices satisfying a certain property. Well-known examples of such parameters include the cardinality of a minimum vertex cover, a minimum dominating set, or a minimum feedback vertex set in a graph. In many cases requiring the subset of vertices to be {\it connected}, that is, to induce a connected subgraph defines a natural variant of the original parameter. The cardinality of a minimum connected vertex cover or a minimum connected dominating set are just two examples of such parameters that have received considerable interest from both the algorithmic and structural graph theory communities. An interesting question is what effect the additional connectivity constraint has on the value of the graph parameter in question. The {\it price of connectivity} for a certain graph property, for which a connected variant exists, and a class of graphs ${\cal G}$ is defined as the 
worst-case ratio $\pi'(G)/\pi(G)$ over all connected graphs $G\in {\cal G}$, where $\pi(G)$ and $\pi'(G)$ denote the smallest subset and smallest connected subset, respectively, of the vertices of $G$ satisfying the property.

Cardinal and Levy~\cite{CL10} coined the term ``price of connectivity''.  They did this for vertex cover. Let $\vc(G)$ denote the vertex cover number, which is the cardinality of a minimum vertex cover of a graph $G$. The connected variant of this parameter is the connected vertex cover number, denoted by $\cvc(G)$ and defined as the cardinality of a minimum connected vertex cover in $G$.
Cardinal and Levy~\cite{CL10} proved that  the price of connectivity for vertex cover is at most $2/(1+\epsilon)$ for graphs with average degree $\epsilon n$.
Camby, Cardinal, Fiorini and Schaudt~\cite{CCFS14} considered general graphs and proved that for every connected graph $G$, it holds that $\cvc(G)\leq 2\cdot \vc(G) -1$, that is, the price of connectivity for vertex cover is upper bounded by~2 for the class of all graphs. The same authors showed that the bound of~$2$ is asymptotically sharp for the classes of all paths and all cycles. Camby et al.~\cite{CCFS14} also provided forbidden induced subgraph characterizations of graph classes for which the price of connectivity for vertex cover is upper bounded by $t$, for $t\in \{1,4/3,3/2\}$.

Grigoriev and Sitters~\cite{GS09} proved that the price of connectivity for face hitting set is at most~11 for planar graphs of minimum degree at least~3. This upper bound was later reduced to~5 by Schweitzer and Schweitzer~\cite{SS10}, who also proved that their bound is tight.
 
For a graph $G$, the (connected) domination number is the size of a smallest (connected) dominating set of $G$; we denote these numbers by $\ds(G)$ and
$\cds(G)$, respectively. Duchet and Meyniel~\cite{DM82} observed that $\cds(G)\leq 3\cdot \ds(G)-2$ for every connected graph~$G$.
Hence, the price of connectivity for dominating set~on general graphs is upper bounded by~$3$. 
If a graph $G$ has no induced subgraph isomorphic to a graph in $\{H_1,\ldots,H_p\}$ then $G$ is said to be {\it $(H_1,\ldots,H_p)$-free}.
Zverovich~\cite{Zve03} proved that for any graph~$G$, it holds that $\cds(H)=\ds(H)$ for each connected induced subgraph $H$ of $G$ if and only if $G$ is $(P_5,C_5)$-free.
Consequently, the price of connectivity for dominating set is exactly~$1$ for the class of $(P_5,C_5)$-free graphs. Camby and Schaudt~\cite{CS14} proved that $\cds(G)\leq \ds(G)+1$ for every connected $(P_6,C_6)$-free graph $G$, and showed that this bound is best possible. They also obtained an upper bound of~$2$ on the price of connectivity for dominating set for $(P_8,C_8)$-free graphs, and 
showed that this bound is sharp even for $(P_7,C_7)$-free graphs. Moreover,
they showed that the general upper bound of~$3$ is asymptotically sharp for $(P_9,C_9)$-free graphs. 

Camby and Schaudt~\cite{CS14} also considered the problem of deciding whether the price connectivity for dominating set is bounded by  some integer~$r$ and proved 
that this problem is $P^{\mbox{{\small NP}}[\log]}$-complete for every fixed constant $1<r<3$.

\subsection{Our Results}\label{l-ours}
We initiate the study of the price of connectivity for feedback vertex set. A \emph{feedback vertex set} of a graph is a subset of its vertices whose removal yields an acyclic graph, or equivalently, a forest. 
For a graph $G=(V,E)$ and a set $S\subseteq V$ we let $G[S]$ denote the subgraph of $G$ induced by $S$.
A feedback vertex set~$S$ of a graph $G$ is {\it connected} if  $G[S]$ is connected. We write $\fvs(G)$ and $\cfvs(G)$ to denote the cardinalities of a minimum feedback vertex set and a minimum connected feedback vertex set of a graph $G$, respectively. For a class of graphs ${\cal G}$, the {\em price of connectivity for feedback vertex set} (poc-fvs) for ${\cal G}$ is defined to be the maximum ratio $\cfvs(G)/\fvs(G)$ over all connected graphs $G\in {\cal G}$. Graphs consisting of two disjoint cycles that are connected to each other by an arbitrarily long path show that the poc-fvs for general graphs, or even for planar graphs, is not upper bounded by a constant.

Inspired by the work of Camby et al.~\cite{CCFS14}, Zverovich~\cite{Zve03}, and Camby and Schaudt~\cite{CS14}, we focus on graph classes that are characterized by a finite family of forbidden induced subgraphs. Recall that, for a family of graphs ${\cal H}$, a graph $G$ is called ${\cal H}$-free if for every $H\in {\cal H}$, $G$ is $H$-free, that is, if $G$ has no an induced subgraph isomorphic to $H$.  The vast majority of well-studied graph classes have forbidden induced subgraphs characterizations, and such characterizations can often be exploited when proving structural or algorithmic properties of these graph classes. In fact, for every hereditary graph class~${\cal G}$, that is, for every graph class ${\cal G}$ that is closed under taking induced subgraphs, there exists a family ${\cal H}$ of graphs such that ${\cal G}$ is exactly the class of ${\cal H}$-free graphs. Notable examples of graphs classes that can be characterized using a {\em finite} family of forbidden induced subgraphs include claw-free graphs, line graphs, split graphs and cographs.

Our first main result (Theorem~\ref{t-dichotomy2} below) establishes a dichotomy between the finite families ${\cal H}$ for which the poc-fvs for ${\cal H}$-free graphs is upper bounded by a constant $c_{\cal H}$ and the families ${\cal H}$ for which such a constant $c_{\cal H}$ does not exist. 
In the case where $|{\cal H}|=1$, this dichotomy implies that the poc-fvs for $H$-free graphs is bounded by a constant $c_H$ if and only if $H$ is a {\it linear forest}, i.e., a disjoint union of paths. Our second main result (Theorem~\ref{t-tetrachotomy} below) establishes a more refined tetrachotomy result for the case $|{\cal H}|=1$. More precisely, for every graph $H$, we determine which of the following cases holds: 
\begin{enumerate}[(i)]
\item $\cfvs(G) = \fvs(G)$ for every connected $H$-free graph $G$; 
\item there exists a constant $c_H$ such that $\cfvs(G) \leq \fvs(G)+c_H$ for every connected $H$-free graph $G$;
\item there exists a constant $c_H$ such that $\cfvs(G) \leq c_H \cdot \fvs(G)$ for every connected $H$-free graph $G$;
\item none of the above three cases applies. 
\end{enumerate}
In order to formally state our results, we need to introduce some terminology. 

For two graphs $H_1$ and $H_2$, we write $H_1+H_2$ to denote the disjoint union of $H_1$ and $H_2$. We write $sH$ to denote the disjoint union of $s$ copies of $H$. For any $k\geq 1$, we write $P_k$ to denote the path on $k$ vertices, that is, the path of length~$k-1$. For any $r\geq 3$, we write $C_r$ to denote the cycle on $r$ vertices. For any three integers $i,j,k$ with $i,j\geq 3$ and $k\geq 1$, we define $B_{i,j,k}$ to be the graph obtained from $C_i+C_j$ by choosing a vertex $x$ in $C_i$ and a vertex $y$ in $C_j$, and adding a path of length $k$ between $x$ and $y$; see Figure~\ref{f-butterfly} for a picture of the graph $B_{5,9,4}$. We call a graph of the form $B_{i,j,k}$ a {\em butterfly}.

\begin{figure}[htb]
\centering
\includegraphics{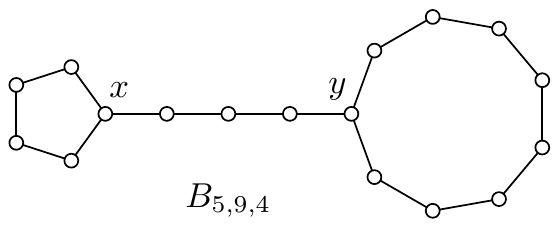}
\caption{The butterfly $B_{5,9,4}$.}
\label{f-butterfly}
\end{figure}

It is clear that the price of connectivity for feedback vertex set for the class of all butterflies is not bounded by a constant, since $\fvs(B_{i,j,k})=2$ and $\cfvs(B_{i,j,k})=k+1$ for every $i,j\geq 3$ and $k\geq 1$. Roughly speaking, our first main result states that the poc-fvs for the class of ${\cal H}$-free graphs is bounded by a constant $c_{\cal H}$ if and only if the forbidden induced subgraphs in ${\cal H}$ prevent arbitrarily large butterflies from appearing as induced subgraphs. To make this statement more concrete, we use the following definition.

\begin{definition}
\label{d-cover}
Let $i,j\geq 3$ be two integers, ${\cal H}$ be a 
finite
family of graphs, and $N= 2 \cdot \max_{H\in {\cal H}} |V(H)| +1$. The family ${\cal H}$ {\em covers} the pair $(i,j)$ if ${\cal H}$ contains an induced subgraph of $B_{i,j,N}$.
A graph $H$ covers the pair $(i,j)$ if $\{H\}$ covers $(i,j)$.
\end{definition}

We are now ready to formally state our first main result, which we prove in Section~\ref{s-1}.

\begin{theorem}
\label{t-dichotomy2}
Let ${\cal H}$ be a finite family of graphs. Then the poc-fvs for the 
class of ${\cal H}$-free graphs is upper bounded by a constant $c_{\cal H}$ if and only if ${\cal H}$ covers the pair $(i,j)$ for every $i,j\geq 3$.
\end{theorem}

\begin{figure}
\begin{center}
\includegraphics[scale=1]{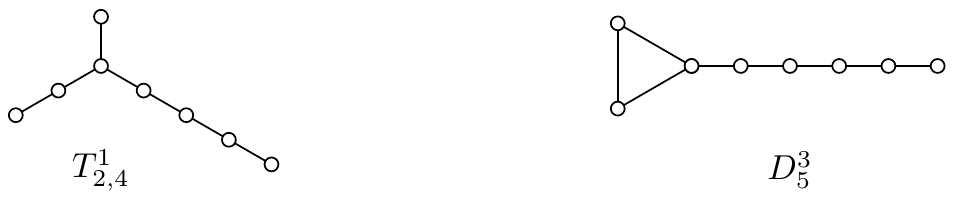}
\caption{The graphs $T^1_{2,4}$ and $D^3_5$.}
\label{f-sijk}
\end{center}
\end{figure}

By using Theorem~\ref{t-dichotomy2} we are able to give obtain an explicit description of exactly those families $\{H_1,H_2\}$ for which the poc-fvs for $\{H_1,H_2\}$-free graphs is upper bounded by a constant. For any $k,p,q\geq 1$, let $T_k^{p,q}$ denote the graph obtained from $P_k+P_p+P_q$ by making a new vertex adjacent to one end-vertex of each path. For any $k\geq 0$ and $r\geq 3$, let $D_k^r$ denote the graph obtained from $P_k+C_r$ by adding an edge between a vertex of the cycle and an end-vertex of the path; in particular, $D_0^r$ is isomorphic to $C_r$. See Figure~\ref{f-sijk} for a picture of the graphs $T^1_{2,4}$ and $D^3_5$.

\begin{corollary}
\label{cor:H1-H2-free}
Let $H_1$ and $H_2$ be two graphs, and let ${\cal H}=\{H_1,H_2\}$. Then the poc-fvs for ${\cal H}$-free graphs is upper bounded by a constant $c_{\cal H}$ if only if there exist integers $\ell\geq 0$ and $r\geq 1$ such that one of the following conditions holds:
\begin{itemize}
\item $H_1$ or $H_2$ is a linear forest;
\item $H_1$ and $H_2$ are induced subgraphs of $D_\ell^3$ and $2T^{1,1}_r$, respectively;
\item $H_1$ and $H_2$ are induced subgraphs of $2D_\ell^3$ and $T^{1,1}_r$, respectively.
\end{itemize}
\end{corollary}
We prove Corollary~\ref{cor:H1-H2-free} in Section~\ref{s-proof2} after first 
proving a sequence of lemmas that show exactly which graphs $H$ cover which pairs $(i,j)$. 
In fact we describe a procedure that, given a positive integer $k$, yields an explicit description of all the 
graph families ${\cal H}$ with $|{\cal H}|=k$ for which the poc-fvs for ${\cal H}$-free graphs is upper bounded by a constant (but we only give an explicit description for $k=2$ as stated in Corollary~\ref{cor:H1-H2-free}). 

As every graph $H$ that is an induced subgraph of both $D_\ell^3$ for some $\ell\geq 0$ and of $2T^{1,1}_r$ for some $r\geq 1$ is a linear forest, Corollary~\ref{cor:H1-H2-free} implies that for any graph~$H$, the poc-fvs for the class of $H$-free graphs is upper bounded by a constant $c_{\cal H}$ if and only if $H$ is a linear forest. Our second main result, which is proven in Section~\ref{s-tetrachotomy}, refines this statement.

\begin{theorem}
\label{t-tetrachotomy}
Let $H$ be a graph. Then it holds that 
\begin{enumerate}[(i)]
\item $\cfvs(G) = \fvs(G)$ for every connected $H$-free graph $G$ if and only if $H$ is an induced subgraph of $P_3$;
\item there exists a constant $c_H$ such that $\cfvs(G) \leq \fvs(G)+c_H$ for every connected $H$-free graph $G$ if and only if $H$ is an induced subgraph of $P_5+sP_1$ or $sP_3$ for some $s\geq 0$;
\item there exists a constant $c_H$ such that $\cfvs(G) \leq c_H \cdot \fvs(G)$ for every connected $H$-free graph $G$ if and only if $H$ is a linear forest.
\end{enumerate}
\end{theorem}

\subsection{Future Work}
\label{s-conclusion}

A natural question to ask is whether Theorem~\ref{t-dichotomy2} can be extended to families ${\cal H}$ which are not finite, i.e., to all hereditary classes of graphs. Definition~\ref{d-cover} and Theorem~\ref{t-dichotomy2} show that for any finite family ${\cal H}$, the poc-fvs for ${\cal H}$-free graphs is bounded essentially when the graphs in this class do not contain arbitrarily large induced butterflies. The following example shows that when ${\cal H}$ is infinite, it is no longer only butterflies that can cause the poc-fvs to be unbounded. Let $G$ be a graph obtained from $K_3$ by first duplicating every edge once, and then subdividing every edge arbitrarily many times. Let ${\cal G}$ be the class of all graphs that can be constructed this way. In order to make ${\cal G}$ hereditary, we take its closure under the induced subgraph relation. Let ${\cal G}'$ be the resulting graph class. Observe that graphs in this class have arbitrarily large minimum connected feedback vertex sets, while $\fvs(G)\leq 2$ for every graph $G\in {\cal G}'$. Hence, the poc-fvs for ${\cal G}'$ is not bounded. However, no graph in this family contains a butterfly as an induced subgraph.

Another natural question to ask is how to extend Theorem~\ref{t-tetrachotomy}  to all finite families of graphs ${\cal H}$.

Also, instead of demanding that the graph obtained after removing some subset of vertices is an independent set or a forest, as in the case of vertex
cover or feedback vertex set, respectively, we may impose other restrictions on the graph induced by the remaining vertices.
Recently, Hartinger et al.~\cite{HJMP15} obtained several results for the price of connectivity in this direction but many open problems remain (see~\cite{HJMP15} for details).

\section{Proof of Theorem~\ref{t-dichotomy2}}\label{s-1}
\label{s-proof}

In order to prove Theorem~\ref{t-dichotomy2} we will make use of the following two observations.

\begin{observation}
\label{o-smallsubgraphs}
Let $i,j,k,\ell$ be integers such that $i,j\geq 3$ and $\ell\geq k\geq 1$. A graph on at most $k$ vertices is an induced subgraph of $B_{i,j,k}$ if and only if it is an induced subgraph of $B_{i,j,\ell}$. 
\end{observation}

\begin{observation}
\label{o-nicefvs}
Let $G$ be a connected graph that is not a cycle. Then $G$ has a minimum feedback vertex set $F$ such that each vertex in $F$ lies on a cycle and has degree at least~$3$ in $G$.
\end{observation}

We are now ready to present the proof of Theorem~\ref{t-dichotomy2}, which we restate below.

\medskip
\noindent
{\bf Theorem~\ref{t-dichotomy2}.}
{\it Let ${\cal H}$ be a finite family of graphs. Then the poc-fvs for  for the class of ${\cal H}$-free graphs is upper bounded by a constant $c_{\cal H}$ if and only if ${\cal H}$ covers the pair $(i,j)$ for every $i,j\geq 3$.}

\begin{proof}
First suppose there exists a pair $(i,j)$ with $i,j\geq 3$ such that ${\cal H}$ does not cover $(i,j)$. For contradiction, suppose there exists a constant $c_{\cal H}$ as in the statement of the theorem. By Definition~\ref{d-cover}, ${\cal H}$ does not contain an induced subgraph of $B_{i,j,N}$, and hence $B_{i,j,N}$ is ${\cal H}$-free. As a result of Observation~\ref{o-smallsubgraphs}, $B_{i,j,k}$ is ${\cal H}$-free for every $k\geq N$. In particular, the graph $B_{i,j,N+2c_{\cal H}}$ is ${\cal H}$-free. Note that $\fvs(B_{i,j,N+2c_{\cal H}})=2$ and $\cfvs(B_{i,j,N+2c_{\cal H}})=N+2c_{\cal H}+1$. This implies that $\cfvs(B_{i,j,N+2c_{\cal H}})>c_{\cal H} \cdot \fvs(B_{i,j,N+2c_{\cal H}})$, yielding the desired contradiction. 

\medskip
\noindent
Now suppose that ${\cal H}$ covers the pair $(i,j)$ for every $i,j\geq 3$. Let $G$ be a connected ${\cal H}$-free graph. Observe that $\cfvs(G)=\fvs(G)$ if $G$ is a cycle or a tree, so we assume that $G$ is neither a cycle nor a tree.
We also assume without loss of generality that $G$ does not contain any vertex of degree~1. This can be seen as follows.
Suppose $G$ has a vertex~$u$ of degree~1. Let $G'$ be the graph obtained from $G$ after removing $u$.
Then $\fvs(G')=\fvs(G)$ and $\cfvs(G')=\cfvs(G)$, because $u$ is contained neither in any minimum feedback vertex set nor in any minimum connected feedback vertex set of $G$ and adding $u$ to $G'$ creates no cycles.
 
Below, we will prove that $G$ has
diameter at most $4N$, that is, the distance in $G$ between any two vertices is at most $4N$. To see why this suffices to prove the theorem, observe that we can transform any feedback vertex set~$S$ of~$G$ into a connected feedback vertex set of $G$ of size at most $4N\cdot |S|=4N\cdot \fvs(G)$ by choosing an arbitrary vertex $x\in S$ and adding, for each $y\in S\setminus \{x\}$, all the internal vertices of a shortest path between $x$ and~$y$.

For contradiction, suppose that the diameter of~$G$ is at least $4N+1$. Pick two vertices $x$ and $y$ that are of maximum distance from each other in~$G$. 
Let $A$ and $B$ be the sets that consist of all vertices of $G$ that are of distance at most~$N$ of $x$ and $y$, respectively. Note that $A\cap B=\emptyset$. 
We distinguish three cases.

\medskip
\noindent
{\it Case 1. Both $G[A]$ and $G[B]$ contain a cycle.}\\
Let $C$ be an induced cycle in $G[A]$ and $D$ be an induced cycle in $G[B]$. 
Let $P=u_1\cdots u_p$ be a shortest path between a vertex $u_1\in C$ and a vertex $u_p\in D$. Note that $p\geq 2N$.  
Because $P$ is a shortest path, no vertex of $\{u_3,\ldots,u_{p-1}\}$ is adjacent to a vertex of $C\cup D$.
Moreover, we may assume without loss of generality that $u_1$ is the only neighbour of $u_2$ on $C$ and that $u_p$ is the only neighbour of $u_{p-1}$ on $D$.
In order to see this, suppose this does not hold for $u_2$. We let $v\in C$ be the neighbour of $u_2$ closest to $u_1$ on $C$.  Then we can replace $C$ by a cycle~$C'$ that consists of $u_1,u_2,v$ and a
path between $u_1$ and $v$ on $C$.
We conclude that the vertices on $C$, $D$ and $P$ together induce a butterfly $B_{i,j,N'}$ for some $i,j\geq 3$ and $N'\geq N$.
Since ${\cal H}$ covers the pair $(i,j)$, there exists a graph $H\in {\cal H}$ such that $H$ is an induced subgraph of $B_{i,j,N}$ by Definition~\ref{d-cover}. Due to Observation~\ref{o-smallsubgraphs}, $H$ is also an induced subgraph of $B_{i,j,N'}$ and hence also of $G$. This contradicts the assumption that $G$ is ${\cal H}$-free.

\medskip
\noindent
{\it Case 2. Both $G[A]$ and $G[B]$ are trees.}\\
Since $G$ does not contain any vertex of degree~1, every leaf of $G[A]$ is at distance exactly $N$ from $x$. Since $x$ has degree at least~3, we conclude that $G[A]$ contains an induced $T_N^{N,N}$.
By the same arguments, $G[B]$ contains an induced $T_N^{N,N}$. 
Since~${\cal H}$ covers the pair $(2N,2N)$, there exists a graph $H\in {\cal H}$ such that $H$ is an induced subgraph of $B_{2N,2N,N}$. This graph $H$ has at most $N$ vertices, which implies that~$H$ has no cycle. However, then $H$ is an induced subgraph of $G[A\cup B]$, and thus of $G$, yielding the same contradiction as in Case~1.

\medskip
\noindent
{\it Case 3. Cases 1-2 do not apply.}\\
Then, as $G[A]$ and $G[B]$ are connected, we may assume without loss of generality that $G[A]$ contains a cycle and $G[B]$ is a tree.
By the arguments of the previous two cases we find that $G$ contains an induced $D_{3N}^i$ for some $i\geq 3$.
Since ${\cal H}$ covers the pair $(i,3N)$, there exists a graph $H\in {\cal H}$ such that $H$ is an induced subgraph of $B_{i,3N,N}$. Since $|V(H)|\leq N$, we find that $H$ contains at most one cycle, and this cycle, if it exists, is of length $i$. Hence, $H$ is an induced subgraph of $G[A\cup B]$ and thus of $G$. With this contradiction we have completed
the proof of Theorem~\ref{t-dichotomy2}.
\qed
\end{proof}

\section{Proof of Corollary~\ref{cor:H1-H2-free}}
\label{s-proof2}

Recall that by Definition~\ref{d-cover} a graph $H$ covers a pair $(i,j)$ if and only if $H$ is an induced subgraph of $B_{i,j,N}$, where $N=2\cdot |V(H)|+1$. In particular, if a graph~$H$ is not an induced subgraph of a butterfly, then $H$ does not cover any pair $(i,j)$. 
For convenience, we describe all the possible induced subgraphs of $B_{i,j,N}$ in the following observation.

\begin{observation}
\label{obs:induced-subgraphs}
Let $H$ be a graph, let $N=2\cdot |V(H)|+1$, and let $i,j\geq 3$ be two integers. Then $H$ is an induced subgraph of $B_{i,j,N}$ if and only if $H$ is isomorphic to the disjoint union of a linear forest (possibly on zero vertices) and at most one of the following graphs:
\begin{enumerate}[(i)]
\item $D_\ell^i$ for some $\ell\geq 0$;
\item $D_\ell^j$ for some $\ell\geq 0$;
\item $D_\ell^i+D_{\ell'}^j$ for some $\ell,\ell'\geq 0$;
\item $T_k^{p,q}$ for some $k,p,q\geq 1$ such that $p+q+2\leq \max\{i,j\}$;
\item $T_k^{p,q}+T_{k'}^{p',q'}$ for some $k,p,q,k',p',q'\geq 1$ such that $p+q+2\leq i$ and $p'+q'+2\leq j$;
\item $D_\ell^i+T_{k}^{p,q}$ for some $\ell\geq 0$ and $k,p,q\geq 1$ such that $p+q+2\leq j$;
\item $D_\ell^j+T_{k}^{p,q}$ for some $\ell\geq 0$ and $k,p,q\geq 1$ such that $p+q+2\leq i$.
\end{enumerate}
\end{observation}

In order to prove Corollary~\ref{cor:H1-H2-free} we will give a series of lemmas that, 
for each of the induced subgraphs described in Observation~\ref{obs:induced-subgraphs}, show exactly which pairs~$(i,j)$ they cover. 
It is important to note that some induced subgraphs of $B_{i,j,N}$ cover more pairs than others. For example, as we will see in Lemma~\ref{lem:linear-forest}, a linear forest covers all pairs $(i,j)$ with $i,j\geq 3$, but this is not the case for any induced subgraph of $B_{i,j,N}$ that is not a linear forest. 
At the end of the proof of each of the lemmas, we refer to a table in which the set of covered pairs is depicted. 

\begin{lemma}
\label{lem:one-cycle}
Let $H$ be a graph, let $p\geq 3$, and let ${\cal X}$ be the set consisting of the pairs $(i,j)$ with $i,j\geq 3$ and $p\in \{i,j\}$. 
\begin{enumerate}[(i)]
\item If $H$ is an induced subgraph of $D_k^{p}$ for some $k \geq 0$, then $H$ covers all the pairs in~${\cal X}$.
\item If $D_k^p$ is an induced subgraph of $H$ for some $k\geq 0$, then $H$ does not cover any pair that does not belong to ${\cal X}$.
\end{enumerate}
\end{lemma}

\begin{proof}
Let $N=2\cdot |V(H)|+1$. Suppose $H$ is an induced subgraph of $D_k^p$ for some $k\geq 0$. Then $H$ is also an induced subgraph of $B_{i,j,N}$ for every $i,j\geq 3$ such that $p\in \{i,j\}$. Hence, by Definition~\ref{d-cover}, $H$ covers the pairs $(p,j)$ and $(i,p)$ for every $i,j\geq 3$.

Now suppose $D_k^p$ is an induced subgraph of $H$ for some $k\geq 0$. Then $H$ contains a cycle of length~$p$. Hence it is clear that if $H$ is an induced subgraph of a butterfly $B_{i,j,N}$, then we must have $p\in \{i,j\}$. This shows that $H$ does not cover any pair that does not belong to ${\cal X}$.

See the left table in Figure~\ref{f-onecycle_twocycles} for an illustration of the pairs in ${\cal X}$.
\qed
\end{proof}

\begin{lemma}
\label{lem:two-cycles}
Let $H$ be a graph, let $p,q \geq 3$, and let ${\cal X}=\{(p,q),(q,p)\}$. 
\begin{enumerate}[(i)]
\item If $H$ is an induced subgraph of $D_k^{p} + D_k^{q}$ for some $k \geq 0$, then $H$ covers all the pairs in ${\cal X}$.
\item If $D_k^{p} + D_k^{q}$ is an induced subgraph of $H$ for some $k \geq 0$, then $H$ does not cover any pair that does not belong to ${\cal X}$.
\end{enumerate}
\end{lemma}

\begin{proof}
Let $N=2\cdot |V(H)|+1$. First suppose $H$ is an induced subgraph of $D_k^{p} + D_k^{q}$ for some $k \geq 0$. Then $H$ is also an induced subgraph of $B_{p,q,N}$ (use Observation~\ref{o-smallsubgraphs} if $k>|V(H)|$).
Hence, by Definition~\ref{d-cover}, graph $H$ covers the pairs $(p,q)$ and $(q,p)$.

To prove (ii), suppose $D_k^{p} + D_k^{q}$ is an induced subgraph of $H$ for some $k\geq 0$. Then $H$ contains a cycle of length~$p$ and a cycle of length~$q$. Hence, from the definition of butterflies and by Definition~\ref{d-cover}, it is clear that $H$ is not an induced subgraph of $B_{i,j,N}$ for any $i,j\geq 3$ such that $\{p,q\}\neq \{i,j\}$. This proves~(ii).

See the right table in Figure~\ref{f-onecycle_twocycles} for an illustration of the pairs in ${\cal X}$.
\qed
\end{proof}

\renewcommand{\arraystretch}{1.6} % for increasing row heigth  
\newcolumntype{D}{ >{\centering\arraybackslash} m{1.35cm} } % for vertically-centered text in leftmost column
\newcolumntype{C}{ >{\centering\arraybackslash} m{.6cm} } % for vertically-centered text in other columns

\begin{figure}[!htb]
\centering
 \begin{minipage}{.45\textwidth}
  \begin{tabular}{D||C|C|C|C|C|C}
    \backslashbox[\dimexpr\linewidth+2\tabcolsep]{$~i$}{$j~$} & 3 & {\tiny $~\cdots$} & {\tiny $~\cdots$} & $p$ & {\tiny $~\cdots$} & {\tiny $~\cdots$}\\
    \hline\hline
    $~3$  & & & & $\checkmark$ & & \\
    \hline
    {\tiny $~\vdots$}  & & & & $\checkmark$ & &  \\ 
    \hline
    {\tiny $~\vdots$}  & & & & $\checkmark$ & & \\ 
    \hline
    $~p$ & $\checkmark$  &   $\checkmark$  & $\checkmark$ & $\checkmark$ & $\checkmark$ & $\checkmark$  \\
    \hline
    {\tiny $~\vdots$}  & & & & $\checkmark$ & &  \\
    \hline 
    {\tiny $~\vdots$}  & & & & $\checkmark$ & & \\
  \end{tabular}\\[2mm]
 \end{minipage}
 \hspace{.6cm}
 \begin{minipage}{.45\textwidth}
  \centering
    \begin{tabular}{D||C|C|C|C|C|C}
    \backslashbox[\dimexpr\linewidth+2\tabcolsep]{$~i$}{$j~$} & 3 & {\tiny $~\cdots$} & $p$ & {\tiny $~\cdots$} & $q$ & {\tiny $~\cdots$}\\
    \hline\hline
    $~3$  & & & & & & \\
    \hline
    {\tiny $~\vdots$}  & & & & & & \\
    \hline
    $~p$ & & & & & $\checkmark$ &  \\
    \hline
    {\tiny $~\vdots$} & & & & & & \\
    \hline
    $~q$ & & & $\checkmark$ & & & \\
    \hline
    {\tiny $~\vdots$} & & & & & & \\
  \end{tabular}\\[2mm]
 \end{minipage}
  \caption{The ticked cells represent the pairs $(i,j)$ covered by $H$ when $H$ is isomorphic to $D^p_k$ for some $k\geq 0$ (left table) and when $H$ is isomorphic to $D^p_k + D^q_k$ for some $k\geq 0$ (right table).}
 \label{f-onecycle_twocycles}
\end{figure}

\begin{lemma}
\label{lem:one-star}
Let $H$ be a graph, let $p,q\geq 1$, and let ${\cal X}$ be the set consisting of the pairs $(i,j)$ with $i,j\geq 3$ and $\max\{i,j\} \geq p+q+2$. 
\begin{enumerate}[(i)]
\item If $H$ is an induced subgraph of $T_r^{p,q}$ for some $r \geq 1$, then $H$ covers all the pairs in ${\cal X}$.
\item If  $T_r^{p,q}$ is an induced subgraph of $H$ for some $r \geq 1$, then $H$ does not cover any pair that does not belong to ${\cal X}$.
\end{enumerate}
\end{lemma}

\begin{proof}
Let $N=2\cdot |V(H)|+1$. Suppose $H$ is an induced subgraph of $T_r^{p,q}$ for some $r\geq 1$. Then $H$ is also an induced subgraph of the butterfly $B_{i,j,N}$ for every $i,j\geq 3$ such that $\max\{i,j\}\geq p+q+2$. Hence $H$ covers all the pairs in ${\cal X}$ due to Definition~\ref{d-cover}.

For the converse direction, suppose $T_r^{p,q}$ is an induced subgraph of $H$ for some $r\geq 1$. Then $H$ is not an induced subgraph of $B_{i,j,N}$ whenever $\max\{i,j\} < p+q+2$. This shows that $H$ does not cover any pair that does not belong to ${\cal X}$.

See the left table in Figure~\ref{f-onestar_twostars} for an illustration of the pairs in ${\cal X}$.
\qed
\end{proof}

\begin{lemma}
\label{lem:two-stars}
Let $H$ be a graph, let $p,q,p',q'\geq 1$ be such that $p+q \leq p'+q'$, and let ${\cal X}$ consist of all the pairs $(i,j)$ with $\min\{i,j\} \geq p+q+2$ and $\max\{i,j\} \geq p'+q'+2$.
\begin{enumerate}[(i)]
\item If $H$ is an induced subgraph of $T_r^{p,q} + T_{r}^{p',q'}$ for some $r\geq 1$, then $H$ covers all the pairs in ${\cal X}$.
\item If $T_r^{p,q} + T_{r}^{p',q'}$ is an induced subgraph of $H$ for some $r\geq 1$, then $H$ does not cover any pair that does not belong to ${\cal X}$.
\end{enumerate}
\end{lemma}

\begin{proof}
Let $N=2\cdot |V(H)|+1$. Suppose $H$ is an induced subgraph of $T_r^{p,q} + T_{r}^{p',q'}$ for some $r\geq 1$. Then $H$ is also an induced subgraph of $B_{i,j,N}$ for any $i,j$ with $\min\{i,j\} \geq p+q+2$ and $\max\{i,j\} \geq p'+q'+2$. Hence $H$ covers all the pairs in ${\cal X}$.

To prove (ii), suppose $T_r^{p,q} + T_{r}^{p',q'}$ is an induced subgraph of $H$ for some $r\geq 1$. Then $H$ is not an induced subgraph of $B_{i,j,N}$ whenever $\min\{i,j\}<p+q+2$ or $\max\{i,j\}<p'+q'+2$. Hence, by Definition~\ref{d-cover}, $H$ cannot cover any pair that does not belong to ${\cal X}$.

See the right table in Figure~\ref{f-onestar_twostars} for an illustration of the pairs in ${\cal X}$.
\qed
\end{proof}

\begin{figure}[!htb]
\centering
 \begin{minipage}{.45\textwidth}
  \begin{tabular}{D||C|C|C|C|C|C}
    \backslashbox[\dimexpr\linewidth+2\tabcolsep]{$~i$}{$j~$} & 3 & {\tiny $~\cdots$} & {\tiny $~\cdots$} & \begin{sideways}$p+q+2$\end{sideways} & {\tiny $~\cdots$} & {\tiny $~\cdots$}\\
    \hline\hline
    $~3$  & & & & $\checkmark$ & $\checkmark$ & $\checkmark$ \\
    \hline
    {\tiny $~\vdots$}  & & & & $\checkmark$ & $\checkmark$ & $\checkmark$ \\ 
    \hline
    {\tiny $~\vdots$}  & & & & $\checkmark$ & $\checkmark$ & $\checkmark$ \\ 
    \hline
    $~p+q+2~$ & $\checkmark$  &   $\checkmark$  & $\checkmark$ & $\checkmark$ & $\checkmark$ & $\checkmark$  \\
    \hline
    {\tiny $~\vdots$}  & $\checkmark$  &   $\checkmark$ & $\checkmark$ & $\checkmark$ & $\checkmark$ & $\checkmark$ \\
    \hline 
    {\tiny $~\vdots$}  & $\checkmark$  &   $\checkmark$ & $\checkmark$ & $\checkmark$ & $\checkmark$ & $\checkmark$ \\
  \end{tabular}\\[2mm]
 \end{minipage}
 \hspace{.6cm}
 \begin{minipage}{.45\textwidth}
  \centering
    \begin{tabular}{D||C|C|C|C|C|C}
    \backslashbox[\dimexpr\linewidth+2\tabcolsep]{$~i$}{$j~$} & 3 & {\tiny $~\cdots$} & \begin{sideways}$p+q+2$\end{sideways} & {\tiny $~\cdots$} & \begin{sideways}$p'+q'+2$\end{sideways} & {\tiny $~\cdots$}\\
    \hline\hline
    $~3$  & & & & & & \\
    \hline
    {\tiny $~\vdots$}  & & & & & & \\
    \hline
    $~p+q+2~$ & & & & & $\checkmark$ & $\checkmark$ \\
    \hline
    {\tiny $~\vdots$} & & & & & $\checkmark$ & $\checkmark$ \\
    \hline
    $~p'+q'+2~$ & & & $\checkmark$ & $\checkmark$ & $\checkmark$ & $\checkmark$ \\
    \hline
    {\tiny $~\vdots$} & & & $\checkmark$ & $\checkmark$ & $\checkmark$ & $\checkmark$ \\
  \end{tabular}\\[2mm]
 \end{minipage}
  \caption{The ticked cells represent the pairs $(i,j)$ covered by $H$ when $H$ is isomorphic to $T_r^{p,q}$ for some $r\geq 1$ (left table) and when $H$ is isomorphic to $T_r^{p,q} + T_r^{p',q'}$ for some $r\geq 1$ (right table).}
 \label{f-onestar_twostars}
\end{figure}

\begin{lemma}
\label{lem:mixed}
Let $H$ be a graph, let $p\geq 3$ and $p',q'\geq 1$, and let ${\cal X}$ be the set consisting of the pairs $(i,j)$ with either $p=i$ and $j\geq p'+q'+2$ or $p=j$ and $i\geq p'+q'+2$.
\begin{enumerate}[(i)]
\item If $H$ is an induced subgraph of $D_k^{p} + T_r^{p',q'}$ for some $k\geq 0$ and $r\geq 1$, then $H$ covers all the pairs in ${\cal X}$.
\item If $D_k^{p} + T_r^{p',q'}$ is an induced subgraph of $H$ for some $k\geq 0$ and $r\geq 1$, then $H$ does not cover any pair that does not belong to ${\cal X}$.
\end{enumerate}
\end{lemma}

\begin{proof}
Let $N=2\cdot |V(H)|+1$. Suppose $H$ is an induced subgraph of 
$D_k^{p} + T_r^{p',q'}$ for some $k\geq 0$ and $r\geq 1$. Then $H$ is an induced subgraph of $B_{i,j,N}$ for every $i,j$ with either $p=i$ and $j\geq p'+q'+2$ or $p=j$ and $i\geq p'+q'+2$. Hence, by Definition~\ref{d-cover}, $H$ covers all the pairs in ${\cal X}$.

To prove (ii), suppose $D_k^{p} + T_r^{p',q'}$ is an induced subgraph of $H$ for some $k\geq 0$ and $r\geq 1$. 
Suppose $H$ covers a pair $(i,j)$. By Definition~\ref{d-cover}, $H$ is an induced subgraph of $B_{i,j,N}$. Since $H$ contains a cycle of length~$p$ due to the presence of $D_k^p$ as induced subgraph, it holds that $p\in \{i,j\}$. Suppose $p=i$. Since $H$ contains $T_k^{p',q'}$ as an induced subgraph, we must have that $j\geq p'+q'+2$. Similarly, if $p=j$, then it holds that $i\geq p'+q'+2$. We conclude that $(i,j)\in {\cal X}$, which suffices to prove (ii).

See the left and right tables in Figure~\ref{f-onecycleonestar_cases1and2} and the table in Figure~\ref{f-onecycleonestar_case3} for an illustration of the pairs in ${\cal X}$ in the cases where $p<p'+q'+2$, $p>p'+q'+2$, and $p=p'+q'+2$, respectively.
\qed
\end{proof}

\begin{figure}[!htb]
\centering
 \begin{minipage}{.45\textwidth}
    \begin{tabular}{D||C|C|C|C|C|C}
    \backslashbox[\dimexpr\linewidth+2\tabcolsep]{$~i$}{$j~$} & 3 & {\tiny $~\cdots$} & $p$ & {\tiny $~\cdots$} & \begin{sideways}$p'+q'+2$\end{sideways} & {\tiny $~\cdots$}\\
    \hline\hline
    $~3$  & & & & & & \\
    \hline
    {\tiny $~\vdots$}  & & & & & & \\
    \hline
    $~p$ & & & & & $\checkmark$ & $\checkmark$ \\
    \hline
    {\tiny $~\vdots$} & & & & & & \\
    \hline
    $~p'+q'+2~$ & & & $\checkmark$ & & & \\
    \hline
    {\tiny $~\vdots$} & & & $\checkmark$ & & & \\
  \end{tabular}\\[2mm]
 \end{minipage}
 \hspace{.6cm}
 \begin{minipage}{.45\textwidth}
  \centering
    \begin{tabular}{D||C|C|C|C|C|C}
    \backslashbox[\dimexpr\linewidth+2\tabcolsep]{$~i$}{$j~$} & 3 & {\tiny $~\cdots$} & \begin{sideways}$p'+q'+2$\end{sideways} & {\tiny $~\cdots$} & $p$ & {\tiny $~\cdots$}\\
    \hline\hline
    $~3$  & & & & & & \\
    \hline
    {\tiny $~\vdots$}  & & & & & & \\
    \hline
    $~p'+q'+2~$ & & & & & $\checkmark$ &  \\
    \hline
    {\tiny $~\vdots$} & & & & & $\checkmark$ & \\
    \hline
    $~p$ & & & $\checkmark$ & $\checkmark$ & $\checkmark$ & $\checkmark$ \\
    \hline
    {\tiny $~\vdots$} & & & & & $\checkmark$ & \\
  \end{tabular}\\[2mm]
 \end{minipage}
  \caption{The ticked cells represent the pairs $(i,j)$ covered by $H$ when $H$ is isomorphic to $D^p_k+T^{p',q'}_r$ for some $k\geq 0$ and $r\geq 1$ in the case where $p<p'+q'+2$ (left table) and in the case where $p>p'+q'+2$ (right table).}
\label{f-onecycleonestar_cases1and2}
\end{figure}

\begin{figure}[!htb]
 \centering
    \begin{tabular}{D||C|C|C|C|C|C}
    \backslashbox[\dimexpr\linewidth+2\tabcolsep]{$~i$}{$j~$} & 3 & {\tiny $~\cdots$} & {\tiny $~\cdots$} & $p$ & {\tiny $~\cdots$} & {\tiny $~\cdots$}\\
    \hline\hline
    $~3$  & & & & & & \\
    \hline
    {\tiny $~\vdots$}  & & & & & & \\
    \hline
    {\tiny $~\vdots$} & & & & & & \\
    \hline
    $p$ & & & $\checkmark$ & $\checkmark$ & $\checkmark$ & $\checkmark$ \\
    \hline
    {\tiny $~\vdots$}  & & & $\checkmark$ & & & \\
    \hline
    {\tiny $~\vdots$} & & & $\checkmark$ & & & \\
  \end{tabular}\\[2mm]
\caption{The ticked cells represent the pairs $(i,j)$ covered by $H$ when $H$ is isomorphic to $D^p_k+T^{p',q'}_r$ for some $k\geq 0$ and $r\geq 1$ in the case where $p=p'+q'+2$.}
\label{f-onecycleonestar_case3}
\end{figure}

\begin{lemma}
\label{lem:linear-forest}
A graph $H$ covers every pair $(i,j)$ with $i,j\geq 3$ if and only if $H$ is a linear forest.
\end{lemma}

\begin{proof}
If $H$ is a linear forest, then $H$ is an induced subgraph of a path on $2\cdot |V(H)|-1$ vertices. Hence $H$ is also an induced subgraph of $B_{i,j,N}$ for every $i,j\geq 3$, where $N=2\cdot |V(H)|+1$. By Definition~\ref{d-cover}, $H$ covers every pair $(i,j)$ with $i,j\geq 3$.

For the reverse direction, suppose $H$ covers every pair $(i,j)$ with $i,j\geq 3$. For contradiction, suppose $H$ is not a linear forest. Then, as a result of Definition~\ref{d-cover} and Observation~\ref{obs:induced-subgraphs}, either $H$ contains $T^{p,q}_r$ as an induced subgraph for some $p,q,r\geq 1$, or $H$ contains $D^p_k$ as an induced subgraph for some $p\geq 3$ and $k\geq 0$. In the first case, it follows from Lemma~\ref{lem:one-star}(ii) that $H$ does not cover the pair $(3,3)$. In the second case, it follows from Lemma~\ref{lem:one-cycle}(ii) that $H$ does not cover any pair $(i,j)$ with $r \not\in \{i,j\}$. In both cases, we obtain the desired contradiction.
\qed
\end{proof}

Consider the infinite table containing all the pairs $(i,j)$ with $i,j\geq 3$. From Lemmas~\ref{lem:one-star}--\ref{lem:mixed} and Tables~1--7, we can observe two important facts. First, the only graphs $H$ that cover the pair $(3,3)$ are induced subgraphs of $2D_\ell^3$ for some $\ell\geq 0$. Second, the only graphs $H$ that cover infinitely many rows and columns of this table are induced subgraphs of $T_r^{p,q}+T_r^{p',q'}$ for some $r,p,q,p',q'\geq 1$. Hence, any finite family ${\cal H}$ that covers all pairs $(i,j)$ must contain at least one graph of both types. Formally, we have the following observation.

\begin{observation}
\label{obs:must-2T}
Let ${\cal H}$ be a finite family of graphs.
If the poc-fvs for ${\cal H}$-free graphs is upper bounded by a constant $c_{\cal H}$, then ${\cal H}$ contains an induced subgraph of $2D_\ell^3$ for some $\ell\geq 0$ and an induced subgraph of $T^{p,q}_r + T^{p',q'}_r$ for some $r,p,q,p',q'\geq 1$.
\end{observation}

Suppose ${\cal H}$ is a family of  
graphs such that the poc-fvs for ${\cal H}$-free graphs is bounded by a constant. By Observation~\ref{obs:must-2T}, ${\cal H}$ contains a graph $H$ that is an induced subgraph of $T^{p,q}_r + T^{p',q'}_r$ for some  $r,p,q,p',q'\geq 1$. 
If $H$ is also an induced subgraph of $T_r^{p,q}$ for some $r,p,q\geq 1$, or if ${\cal H}$ contains another graph that is of this form, then Lemma~\ref{lem:one-star} and Table~3 show that there are only finitely many pairs $(i,j)$ that are not covered by $H$. These cells need to be covered by the remaining graphs in ${\cal H}$. Using Lemmas~\ref{lem:one-star}--\ref{lem:mixed}, we can determine exactly which combination of graphs covers those remaining pairs.

Suppose ${\cal H}$ does not contain an induced subgraph of $T_r^{p,q}$ for any $r,p,q\geq 1$. Then Lemma~\ref{lem:two-stars} and Table~4 imply that there are finitely many rows and columns in which no pair is covered by $H$. In particular, since $p,q,p',q'\geq 1$, the pairs $(i,3)$ and $(3,j)$ are not covered for any $i,j\geq 3$. From the lemmas in Section~\ref{s-proof2} and the corresponding tables, it it clear that the only graphs  that cover infinitely many pairs of this type are induced subgraphs of $T_r^{p,q}$ for some $r,p,q\geq 1$ or of $D_{r'}^3+T_r^{p,q}$ for some $r'\geq 0$ and $p,q\geq 1$. Hence, ${\cal H}$ must contain a graph that is isomorphic to such an induced subgraph. Similarly, if the pairs $(i,4)$ and $(4,j)$ are not covered for any $i,j\geq 3$, then ${\cal H}$ must contain an induced subgraph of $T_r^{p,q}$ for some $r,p,q\geq 1$ or of $D_{r'}^4+T_r^{p,q}$ for some $r'\geq 0$ and $p,q\geq 1$, and so on. Once all rows and columns contain only finitely many pairs that are not covered yet, we can determine all possible combinations of graphs that cover those last pairs.

We are now ready to restate and prove Corollary~\ref{cor:H1-H2-free}, which can be seen as an illustration of the above procedure
for the case where $|{\cal H}|=2$.

\medskip
\noindent
{\bf Corollary~\ref{cor:H1-H2-free}.}
{\it Let $H_1$ and $H_2$ be two graphs, and let ${\cal H}=\{H_1,H_2\}$. Then the poc-fvs for ${\cal H}$-free graphs is upper bounded by a constant $c_{\cal H}$ if only if there exist integers $\ell\geq 0$ and $r\geq 1$ such that one of the following conditions holds:
\begin{itemize}
\item $H_1$ or $H_2$ is a linear forest;
\item $H_1$ and $H_2$ are induced subgraphs of $D_\ell^3$ and $2T^{1,1}_r$, respectively;
\item $H_1$ and $H_2$ are induced subgraphs of $2D_\ell^3$ and $T^{1,1}_r$, respectively.
\end{itemize}
}

\begin{proof}
First suppose that the price of connectivity for feedback vertex set for ${\cal H}$-free graphs is bounded by some constant $c_{\cal H}$, and suppose that neither $H_1$ nor $H_2$ is a linear forest. Due to Observation~\ref{obs:must-2T}, we may without loss of generality assume that $H_1$ is an induced subgraph of $2D_\ell^3$ for some $\ell\geq 0$ and $H_2$ is an induced subgraph of $T^{p,q}_r + T^{p',q'}_r$ for some $r,p,q,p',q'\geq 1$. From Lemmas~\ref{lem:one-cycle} and~\ref{lem:two-cycles} and the assumption that $H_1$ is not a linear forest, it follows that $H_1$ does not cover the pair $(4,4)$. Hence $H_2$ must cover this pair. This, together with Lemma~\ref{lem:two-stars}, implies that $p=q=p'=q'=1$, i.e., $H_2$ is an induced subgraph of $2T_r^{1,1}$ for some $r\geq 1$. 
If $H_1$ is an induced subgraph of $D_{\ell'}^3$ for some $\ell'\geq 0$, then the second condition holds and we are done. 

Suppose $H_1$ is not an induced subgraph of $D_{\ell'}^3$ for any $\ell'\geq 0$.
Then $H_1$ covers only the pair $(3,3)$ due to Lemma~\ref{lem:two-cycles}. This means that all the pairs $(i,j)$ with $i,j\geq 3$ and $3\in \{i,j\}$, apart from $(3,3)$, must be covered by $H_2$. From Lemma~\ref{lem:one-star} and~\ref{lem:two-stars} it is clear that this only holds if $H_2$ is an induced subgraph of $T_{r'}^{1,1}$ for some $r'\geq 1$. Hence the third condition holds.

The converse direction follows by combining Theorem~\ref{t-dichotomy2} and Lemma~\ref{lem:linear-forest} if the first condition holds, Lemmas~\ref{lem:one-cycle} and~\ref{lem:two-stars} if the second condition holds, and Lemmas~\ref{lem:two-cycles} and~\ref{lem:one-star} if the third condition holds.
\qed
\end{proof}

\section{Proof of Theorem~\ref{t-tetrachotomy}}
\label{s-tetrachotomy}

In order to prove Theorem~\ref{t-tetrachotomy} we start with the following lemma.

\begin{lemma}
\label{l-p5}
For every integer $s$, there is a constant $c_s$ such that $\cfvs(G)\leq \fvs(G) + c_s$ for every connected $P_5+sP_1$-free graph $G$.
\end{lemma}

\begin{proof}
First suppose $H$ is an induced subgraph of $P_5$. Let $G$ be a connected $H$-free graph. In particular, $G$ is $P_5$-free. Hence, due to a result by Bacs\'{o} and Tuza~\cite{BT90}, there exists a dominating set $D\subseteq V(G)$ such that $D$ is a clique or $D$ induces a $P_3$ in~$G$. Let $F$ be a minimum feedback vertex set of $G$. Note that $|D\setminus F|\leq 2$ if $D$ is a clique and $|D\setminus F|\leq 3$ if $D$ induces a $P_3$. Since $D$ is a connected dominating set in $G$, the set $F\cup D$ is a connected feedback vertex set of $G$ of size at most $|F|+3$. Hence, we can take $c_H=3$.

Now suppose $H$ is an induced subgraph of $P_5+sP_1$ for some integer $s$. Let $G$ be a connected $H$-free graph. If $G$ is $P_5$-free, then we can take $c_H=3$ due to the above arguments. Suppose $G$ contains an induced path $P$ on~$5$ vertices. Let $I$ be a maximal independent set in the graph obtained from $G$ by deleting the five vertices of $P$ as well as all their neighbors in $G$. Since $G$ is $P_5+sP_1$-free, we know that $|I|\leq s-1$. Note that $V(P)\cup I$ is a dominating set of $G$. 
Recall that Duchet and Meyniel~\cite{DM82} showed that, for every connected graph $G$, it holds that $\cds(G)\leq 3\cdot \ds(G)-2$.
Hence, there is a connected dominating set $D$ in $G$ of size at most $3(|V(P)|+|I|)-2\leq 3s+10$. Let $S$ be a minimum feedback vertex set in $G$. Then $S\cup D$ is a connected feedback vertex set in $G$ of size at most $|S|+3s+10$. Hence, we can take 
$c_H=3s+10$.
This completes the proof of Lemma~\ref{l-p5}.\qed
\end{proof}

We can prove a similar lemma for the case when $H=sP_3$ for some $s\geq 0$. In order to do this we need an additional lemma.

\begin{lemma}\label{l-theclaim}
Let $s\geq 1$ be an integer and let $G$ be 
a connected $sP_3$-free graph with a subset $S\subseteq V(G)$ and an independent set
$U\subseteq V(G)\setminus S$. If there exists a connected component $Z$ of $G[S]$ that contains an induced copy of $(s-1)P_3$, then
 there exists a set  $S'$ with $S\subseteq S'$ of size at most  $|S|+2s-2$ such that
\begin{itemize}
\item [(i)] $G[S']$ has a connected component $Z'$ containing all vertices of $V(Z)\cup (S'\setminus S)$;
\item [(ii)] every vertex of $U'=U\setminus S'$ is adjacent to at most one connected component of $G[S']$ that is not equal to $Z'$;
\item [(iii)] every connected component of $G[S']$ not equal to $Z'$ is adjacent to at most one vertex of $U'$.
\end{itemize}
\end{lemma}

\begin{proof}
Let graph $A$ be the union of the connected components of $G[S]$ not equal to $Z$ that are adjacent to $U$. Note that all other connected components of $G[S]$ not equal to $Z$ are not relevant for satisfying the conditions of the lemma (as they are not adjacent to $U$).

Let $U_1\subseteq U$ be a minimum set of vertices in $U$ such that every connected component of $A$ is adjacent to $U_1$.
By minimality, every vertex in $U_1$ has a private component in $A$, that is, a connected component of $A$ not adjacent to other vertices of $U_1$.
For every $u\in U_1$, fix a private component of $u$ in $A$, and let $A_1$ be the union of all these private components.
 Consider the graph $A_2=A-A_1$. Every connected component of $A_2$ is adjacent to $U_1$ (since this is true for all connected components of $A$).
Let $U_2$ be a minimum set of vertices in $U_1$ such that every connected component of $A_2$ is adjacent to $U_2$.
By minimality, every vertex in $U_2$ has a private component in $A_2$, that is, a connected component of $A_2$ not adjacent to other vertices of $U_2$.
We have that $|U_2|<s$, as otherwise by using the neighbours in the private components
that each vertex in $U_2\subseteq U_1$ has in $A_1$ and in $A_2$, respectively, 
we find that there is an induced $sP_3$ in $G$.

We now move $U_2$ to $Z$. As each vertex $u\in U_2$ is adjacent to a connected component of $A_1$ and to a connected component of $A_2$, the $sP_3$-freeness of $G$ implies
that $u$ is adjacent to $Z$. Hence moving $U_2$ to $Z$ maintains connectivity of $Z$. Moreover, all connected components of $A$ adjacent to $U_2$ also move to $Z$. Hence, as each component in $A_2$ is adjacent to a vertex in $U_2$, afterwards every remaining connected component of $A$ belongs to $A_1$ and $S$ increased in size by at most~$s-1$. In particular all newly added vertices went to $Z$.

Consider now the remaining connected components of $A_1$. Note that each such connected component has a neighbour in $U_1\setminus U_2$. Let $U_3=U\setminus U_1$. It is still possible that some remaining connected component of $A_1$ is adjacent to a vertex of $U_3$ (besides its neighbour in $U_1\setminus U_2$).  Moreover, a vertex in $U_3$ may be adjacent to more than one remaining connected component of $A_1$.  Below we explain how to handle this situation.

 Let $A_3$ be the union of those remaining connected components of $A_1$ that have at least one neighbour in $U_3$. We pick a minimum set $U_4$
of vertices in $U_3$ such that each connected component of $A_3$ is adjacent to $U_4$.
If $|U_4|\ge s$ then at least $s$ vertices in $U_4$ have private neighbours in $A_3$ (private with respect to other vertices in 
$U_4$). In this case we find an induced $sP_3$ in $G$ (where each $P_3$ has its middle vertex in $A_3$, one
end-vertex in $U_1\setminus U_2$ and the other end-vertex in $U_4$), so this is not possible.
Hence $|U_4|<s$.

As $G$ is $sP_3$-free, every vertex in $U_4$ with at least two connected components in $A_3$ is adjacent to $Z$. We put these vertices in a set $W$.
Then the number of connected components of $A_3$ not adjacent to $W$ is equal to $|U_4|-|W|$, as each such a connected component is adjacent to a vertex in $U_4$ that is not adjacent to any other connected component of $A_3$ (as otherwise we would have placed the vertex in $W$). Moreover, each such a connected component is adjacent
to a vertex in $U_1\setminus U_2$. Hence, we find an induced $P_3$, so at least one of the two neighbours in $U_1\setminus U_2$ or $U_4$, respectively, must be adjacent to $Z$ (due to the $sP_3$-freeness of $G$). We put that vertex in $W$ as well. Then $W$ has the following properties: $W$ has size at most $s-1$, each vertex of $W$ is adjacent to $Z$, and each connected component of $A_3$ is adjacent to a vertex of $W$.

We now move all vertices of $W$  to $Z$. This operations maintains connectivity of $Z$. 
Afterwards, all components of $A_3$ belong to $Z$ as well. Hence, in the end every remaining vertex of $U$ is adjacent to at most one component of $A$, and
every remaining component of $A$ is adjacent to at most one vertex of $U$. So, conditions (ii) and (iii) are satisfied.
As all the new vertices went to $Z$, condition (i) is also satisfied. Moreover, also our last operation made $S$ increase in size by at most $s-1$.
Hence the total increase of $S$ is at most $2s-2$, as required.\qed 
\end{proof}
We can now prove the following lemma.

\begin{lemma}
\label{l-sp3}
For every integer $s$, there is a constant $c_s$ such that $\cfvs(G)\leq \fvs(G) + c_s$ for every connected $sP_3$-free graph $G$.
\end{lemma}

\begin{proof}
We prove the lemma by induction on $s$. 
The lemma trivially holds if $s=1$, since a connected $P_3$-free graph is a complete graph and we can simply take $s_1=0$.
Let $G=(V,E)$ be a connected $sP_3$-free graph for some $s\geq 2$. If $G$ is $(s-1)P_3$-free, then $\cfvs(G)\leq \fvs(G) + c_{s-1}$ for some constant $c_{s-1}$ by the induction hypothesis. Suppose $G$ contains an induced subgraph $H$ isomorphic to $(s-1)P_3$,  
and let $v_1,\ldots,v_{s-1}\in V(H)$ be the vertices of degree~$2$ in $H$. Let $Y\subseteq V$ be the set obtained from $V(H)$ by adding, for $i=2,\ldots,s-1$, all the vertices of a shortest path in $G$ from $v_1$ to $v_i$. Since the assumption that $G$ is $sP_3$-free implies that $G$ is $P_{4s-1}$-free, the total number of vertices in $Y$ is at most $|V(H)|+(s-2)(4s-4)\leq 3(s-1)+(s-2)\cdot4s\leq 4s^2-4s$. The graph $G[Y]$ is connected by construction.

Let $S$ be a minimum feedback vertex set of $G$, and let $S'=S\cup Y$. Then we find that
$$|S'|\leq |S|+4s^2-4s.$$  
We note the following.
Since~$G$ is $sP_3$-free and $G[S']$ contains $(s-1)P_3$ as an induced subgraph, every connected component of $G[S']$, apart from the connected component that contains~$Y$, is a complete graph.
Moreover, the graph $G-S'=G[V\setminus S']$ is a forest due to the fact that $S'$ is a feedback vertex set. 
We now prove the following claim.

\medskip
\noindent
{\em Claim 1: The forest $G-S'$ has at most $4s^2$ vertices of degree at least~$3$ in $G-S'$.}

\smallskip
\noindent
To prove Claim~1, we first root each connected component of $G-S'$ at an arbitrary vertex. Let $T$ be the set of vertices in $G-S'$ that have degree at least~$3$, but have no descendant of degree at least~$3$. Each vertex $v\in T$, together with two of its descendants, induces a $P_3$ in $G$. Since $T$ forms an independent set in $G$, these $|T|$ copies of $P_3$ are mutually induced 
in $G$, that is, the union of the vertex sets of these $|T|$ copies induce a $|T|P_3$ in~$G$.
As $G$ is $sP_3$-free, this implies that $|T|\leq s-1$. The only other possible vertices in $G-S'$ that have degree at least~$3$ in~$G-S'$ are ancestors of vertices in $T$. Hence, any vertex of degree at least~$3$ in $G-S'$ lies on a path from a vertex in $T$ to the root of the corresponding connected component. 
There are at most $|T|\leq s-1$ such paths, each of which contains at most $4s-2$ vertices due to the fact that $G$ is $sP_3$-free and hence $P_{4s-1}$-free. We conclude that there are at most $(s-1)(4s-2)\leq 4s^2$ vertices of degree at least~$3$ in $G-S'$. This completes the proof of Claim~1.

\medskip
\noindent
Starting from the feedback vertex set $S'$, we now construct a connected feedback vertex set of $G$ as follows. First, we add to $S'$ all the vertices of $G-S'$ that have degree at least~$3$ in $G-S'$. By Claim~1, this increases the size of $S'$ by at most $4s^2$. After this step, the graph $G-S'$ is a linear forest. This linear forest cannot contain more than~$4s$ vertices of degree~$2$, as otherwise $G-S'$ would contain an induced subgraph isomorphic to $sP_3$. We now add to $S'$ all vertices of~$G-S'$ that have degree~$2$ in $G-S'$. This increases the size of $S'$ by at most~$4s$. After this step, 
we have that $$|S'|\leq |S|+4s^2-4s+4s^2+4s = |S|+8s^2.$$
Moreover, every connected component of the graph $G-S'$ is isomorphic to either $K_1$ or $K_2$. 
We partition the vertices of $G-S'$ into two (possibly empty) independent sets $U_1$ and $U_2$ as follows: $U_1$ contains exactly one vertex from each connected component of $G-S'$ that is isomorphic to $K_2$, and $U_2$ contains all other vertices of $G-S'$. 
Observe that every vertex of $U_1$ has at most one neighbour in $U_2$ and vice versa.

Now let $Z$ denote the vertex set of the connected component of $G[S']$ that contains~$Y$. 
Our goal is to show that we can reduce the number of connected components of $G[S'\setminus Z]$ to zero by adding vertices to $S'$ that connect those connected components to $G[Z]$, such that in the end $Z$ is a connected feedback vertex set. For this purpose we may also swap some vertices in $S'$ with vertices outside $S'$ as long as in the end we have added at most $f(s)$ vertices, where $f$ is some function that only depends on $s$.

As a first step in obtaining the above goal, we apply Lemma~\ref{l-theclaim} twice, once with respect to $U_1$ and once with respect to $U_2$.
As a result, the set $S'$ increases in size 
by at most $4s-4$ vertices; by Lemma~\ref{l-theclaim} all these extra vertices went to~$Z$.
We now have that $$|S'|\leq |S|+8s^2+4s-4=|S|+ 8s^2+4s-4.$$
Let $U_1'\subseteq U_1$ and $U_2'\subseteq U_2$ be the vertices of $U_1$ and $U_2$, respectively, that did not get added to $Z$ and let ${\cal A}$ denote the set of all connected components of $G[S'\setminus Z]$. By the definition of $U_1$ and $U_2$ and Lemma~\ref{l-theclaim}, the following four properties hold:
\begin{itemize}
\item [1.] both $U'_1$ and $U'_2$ are independent sets;
\item [2.] every vertex  of $U'_1$ has at most one neighbour in $U'_2$ and vice versa;
\item [3.] for $i=1,2$, every vertex in $U'_i$ is adjacent to at most one connected component of ${\cal A}$ and vice versa.
\end{itemize}
 
Let $F$ be the graph induced by the union of all vertices in ${\cal A}$ and $U'_1\cup U'_2$, 
that is, let $F=G-Z$. 
Suppose $F$ has a connected component $X$ that contains an induced $P_3$. If $X$ contains vertices from more than one connected component of $G[U'_1\cup U'_2]$, then there is an induced path in $X$ containing all the vertices of all such connected components due to properties 1--3 mentioned earlier. This, together with the fact that $G$ (and hence~$X$) is $P_{4s-1}$-free, implies that $X$ contains at most $4s-2$ vertices of $U'_1\cup U'_2$. 
Moreover, as $G[Z]$ contains an induced $(s-1)P_3$ and $X$ contains an induced $P_3$, there is a vertex in $X$ that is adjacent to $Z$. We now add all the vertices of $X$ to $Z$, and we do the same with all the vertices of any other connected component of $F$ that contains an induced $P_3$; note that the total number of connected components in $F$ that contain an induced $P_3$ is at most $s-1$ due to the $sP_3$-freeness of $G$. Hence, the size of $S'$ is increased by at most $(4s-2)(s-1)=4s^2-6s+2$, so we now have that
$$|S'|\leq |S|+8s^2+4s-4+4s^2-6s+2=|S|+12s^2-2s-2$$ while the connectivity of $Z$ is maintained. For notational convenience, we use $U'_1$ and $U'_2$ to refer to the vertices of $U'_1$ and $U'_2$, respectively, that have not been added to $Z$ during this procedure.
Since $G$ is connected and there are no components in $F$ anymore that contain an induced $P_3$, we find that every $A\in {\cal A}$ is now 
adjacent to a unique connected component of $G[U'_1\cup U'_2]$, 
which we denote by $B_A$. 
Moreover, for each $B_A$ the following holds: $A$ is the unique connected component of ${\cal A}$ to which $B_A$ is adjacent, and due to this and the connectivity of $G$ we find that $B_A$ is adjacent to $Z$. 

In fact, because we got rid of all connected components of $F$ that contain an induced $P_3$, we have arrived at the situation where $G[V(A)\cup V(B_A)]$ is a complete graph for every connected component $A\in {\cal A}$. Let $A\in {\cal A}$, let $x\in V(A)$, and let~$y$ be a vertex of $B_A$ that is adjacent to $Z$. Since $N_G[x]\subseteq N_G[y]$, we find that the set $S'':=(S'\setminus \{x\})\cup \{y\}$ is a feedback vertex set of $G$, such that $|S''|=|S'|=|S|+12s^2-2s-2$ and the number of connected components in $G[S'']$ is one less than the number of connected components in $G[S']$. Therefore, by repeatedly swapping a vertex of a connected component $A\in {\cal A}$ with a vertex of $B_A$ that is adjacent to $Z$, we reduce the number of connected components of $G[S']$ to~$1$. As desired, this leads to a connected feedback vertex set of $G$ of size $|S|+12s^2-2s-2\leq \fvs(G)+12s^2-2s-2$, so we can take $c_s=12s^2-2s-2$. This completes the proof of  Lemma~\ref{l-sp3}.
\qed
\end{proof}

We are now ready to prove Theorem~\ref{t-tetrachotomy}, which we restate below.

\medskip
\noindent
{\bf Theorem~\ref{t-tetrachotomy}.}
{\it Let $H$ be a graph. Then it holds that 
\begin{enumerate}[(i)]
\item $\cfvs(G) = \fvs(G)$ for every connected $H$-free graph $G$ if and only if $H$ is an induced subgraph of $P_3$;
\item there exists a constant $c_H$ such that $\cfvs(G) \leq \fvs(G)+c_H$ for every connected $H$-free graph $G$ if and only if $H$ is an induced subgraph of $P_5+sP_1$ or $sP_3$ for some $s\geq 0$;
\item there exists a constant $c_H$ such that $\cfvs(G) \leq c_H \cdot \fvs(G)$ for every connected $H$-free graph $G$ if and only if $H$ is a linear forest.
\end{enumerate}
}

\begin{proof}
Statement (iii) follows immediately from Corollary~\ref{cor:H1-H2-free} by taking $H_1=H_2$ (as we already observed in see Section~\ref{s-intro}). 
Before proving~(i) we first prove (ii). 
Due to Lemmas~\ref{l-p5} and~\ref{l-sp3} we are left to show that if $H$ is not an induced subgraph of $P_5+sP_1$ or $sP_3$ for any integer $s$, then there is no constant $c_H$ such that $\cfvs(G)-\fvs(G)+c_H$ for every connected $H$-free graph~$G$. 

Let $H$ be a graph that is not an induced subgraph of $P_5+sP_1$ or $sP_3$ for any integer $s$. First suppose $H$ is not a linear forest. Then, by Theorem~\ref{t-tetrachotomy}~(iii) proven above, there does not exist a constant $c$ such that $\cfvs(G)\leq c\cdot \fvs(G)$ for every connected $H$-free graph $G$. This implies that there cannot exist a constant $c_H$ such that $\cfvs(G)\leq \fvs(G) + c_H$ for every connected $H$-free graph~$G$. 

Now suppose $H$ is a linear forest. Since $H$ is not an induced subgraph of $P_5+sP_1$ or $sP_3$ for any integer $s$, it contains $P_6$ or $P_4+P_2$ as an induced subgraph. Consequently, the class of $H$-free graphs is a superclass of the class of $(P_6,P_4+P_2)$-free graphs. Hence, it suffices to find some subclass ${\cal G}$ of $(P_6,P_4+P_2)$-free graphs for which there exists no constant $c_H$ such that $\cfvs(G)\leq \fvs(G) + c_H$ for every connected $G\in {\cal G}$. Below we describe such a class~${\cal G}$.

\begin{figure}
\begin{center}
\includegraphics[scale=1]{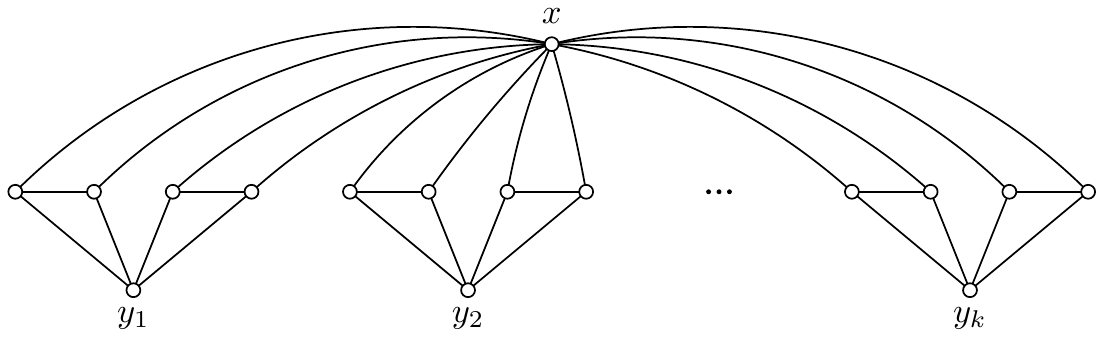}
\caption{The graph $L_k$, defined for every $k\geq 1$.}
\label{f-lk}
\end{center}
\end{figure}

The {\em hourglass} is the graph consisting of two triangles meeting in exactly one vertex. For every integer $k\geq 1$, let $L_k$ be the graph obtained from $k$ disjoint copies of the hourglass by adding a new vertex $x$ that is made adjacent to all vertices of degree~$2$; see Figure~\ref{f-lk} for an illustration. 
Note that $L_k$ is $(P_6,P_4+P_2)$-free for every $k\geq 1$.
For every $k\geq 1$, the unique minimum feedback vertex set in $L_k$ is the set $\{x,y_1,y_2,\ldots,y_k\}$, so $\fvs(L_k)=k+1$. Every minimum connected feedback vertex set in $L_k$ contains the set $\{x,y_1,y_2,\ldots,y_k\}$, as well as exactly one additional vertex for each of the vertices $y_i$ to make this set connected. Hence, $\cfvs(L_k)=2k+1=\fvs(L_k)+k$. Hence, the family $\{L_k\}$ is our desired example.

\medskip
\noindent
We are left to prove~(i).
Let $H$ be a graph. 
First suppose that $H$ is an induced subgraph of $P_3$.
Then $\cfvs(G)=\fvs(G)$ for every connected $H$-free graph $G$, as any such graph is a complete graph.

Now suppose $H$ is not an induced subgraph of $P_3$. 
If $H$ is not an induced subgraph of $P_5+sP_1$ or $sP_3$ for any integer~$s$ then there is no constant~$c_H$ such that $\cfvs(G)\leq \fvs(G) + c_H$ for every connected $H$-free graph $G$  due to Theorem~\ref{t-tetrachotomy}~(ii) proven above. Assume that $H$ is an induced subgraph of $P_5+sP_1$ or $sP_3$ for some integer $s$. 
As $H$ is not an induced subgraph of~$P_3$, it suffices to consider the cases $H=P_1+P_2$ or $H=3P_1$.
If $H=P_1+P_2$, then we let $G=K_{3,\ell}$ (the complete bipartite graph with partition classes of size $3$ and $\ell$, respectively) for some $\ell\geq 3$.
We observe that $G$ is an $H$-free graph with $\cfvs(G)=3$ and $\fvs(G)=2$.
If $H=3P_1$, then we let $G$ be the 6-vertex graph obtained from taking two non-adjacent vertices $u$ and $v$ that we both connect to all vertices of a $P_4$. 
We observe that $G$ is an $H$-free graph with $\cfvs(G)=3$ and $\fvs(G)=2$.
\qed
\end{proof}

\medskip
\noindent
{\it Acknowledgments.} We thank Tatiana Hartinger, Matthew Johnson and Martin Milani{\v c} for helpful comments.

\end{document}